\documentclass[reqno,centertags, 12pt]{amsart}

\usepackage{amssymb,amsmath,amsfonts,amssymb}
-\marginparwidth \oddsidemargin -9mm \evensidemargin \oddsidemargin
\usepackage{latexsym}
\advance\hoffset by 5mm

\def\@abssec#1{\vspace{.05in}\footnotesize \parindent .2in
{\bf #1. }\ignorespaces}
\newtheorem{theorem}{Theorem}[section]

\newtheorem{lemma}[theorem]{Lemma}

\newtheorem{corollary}[theorem]{Corollary}

\def \Rm {\mathbb R}

\def \Zm {\mathbb Z}

\def\cH{\mathcal H}

\allowdisplaybreaks \numberwithin{equation}{section}

\title{Biomixing by chemotaxis and enhancement of biological reactions}
\author{Alexander Kiselev}
\thanks{Department of
Mathematics, University of Wisconsin, Madison, WI 53706, USA;
email: kiselev@math.wisc.edu}
\author{Lenya Ryzhik}
\thanks{Department of
Mathematics, Stanford University, Stanford, CA 94305; email:
ryzhik@math.stanford.edu}

\begin{document}


\begin{abstract}
Many phenomena in biology involve both reactions and chemotaxis.
These processes can clearly influence each other, and chemotaxis can play an important role in
sustaining and speeding up the reaction. However, to the best of our knowledge, the question of reaction
enhancement by chemotaxis has not yet
received extensive treatment either analytically or numerically.
We consider a model with a single density function involving diffusion,
advection, chemotaxis, and absorbing reaction. The model is motivated, in particular,
by studies of coral broadcast spawning, where experimental observations of the efficiency of fertilization
rates significantly exceed the data obtained from numerical models that do not take chemotaxis
(attraction of sperm gametes by a chemical secreted by egg gametes) into account.
We prove that  in the framework of our model, chemotaxis plays a crucial role.
There is a rigid limit to how much the fertilization efficiency can be enhanced if there is no
chemotaxis but only advection and diffusion. On the other hand, when chemotaxis is present,
the fertilization rate can be arbitrarily close to being complete provided that the chemotactic attraction is
sufficiently strong. Moreover, an interesting feature of the estimates on fertilization rate and timescales
in the chemotactic case is that they do not depend on the
amplitude of the reaction term.
\end{abstract}

\maketitle

\section{Introduction}\label{intro}

Our goal in this paper is to study the effect chemotactic attraction may have on reproduction processes in biology. A particular motivation for this study
comes from the phenomenon of broadcast spawning. Broadcast spawning is a fertilization
strategy used by various benthic invertebrates
(sea urchins, anemones, corals) whereby males and females release sperm and egg gametes into the
surrounding flow.
The gametes are positively buoyant, and rise to the surface of the ocean. The sperm and egg are initially separated by the ambient water, and effective mixing is necessary for successful fertilization. The fertilized gametes form larva, which is negatively buoyant and tries to attach to the bottom of the ocean floor to start a new colony.
For the coral spawning problem, field measurements of the fertilization rates are rarely below $5$\%, and are often as high as $90$\% \cite{Eck,Lasker,Penn,Yund}.
On the other hand, numerical simulations based on the
turbulent eddy diffusivity \cite{DShi} predict fertilization rates of less than $1$\% due to the strong dilution of gametes.
The turbulent eddy diffusivity approach involves two scalars that react and diffuse with the
effective diffusivity taking the presence of the flow into account.
It is well known, however, that the geometric structure of the fluid flow lost in the turbulent diffusivity approach
can be important for improving the reaction rate (in the physical and engineering literature see  \cite{Ottino,Ro,Y};
in the mathematical literature see  \cite{MS,CKOR,KR,FKR,KZ} for further references).
Recent work of Crimaldi, Hartford, Cadwell and Weiss \cite{Weiss1,Weiss2} employed a more sophisticated model, taking into account
the instantaneous details of the advective
transport not captured by the eddy diffusivity approach. These papers showed that vortex stirring can generally enhance the reaction rate, perhaps accounting for some of the
discrepancy between the numerical simulations and experiment.

However, there is also experimental evidence that chemotaxis plays a role in coral fertilization: eggs release a chemical that attracts sperm
\cite{Colletal1,Colletal2,Miller1,Miller2}. Mathematically, chemotaxis has been extensively studied in the context of modeling mold and bacterial colonies. Since the original work
of Patlak \cite{Patlak} and Keller-Segel \cite{KS1,KS2} where the first PDE model of chemotaxis was introduced, there has been an enormous amount of effort
devoted to the possible blow up and regularity of solutions, as well as the asymptotic behavior and other properties
(see \cite{Pert} for further references). However, we are not aware
of any rigorous or even computational work on the effects of chemotaxis for improved efficiency of biological reactions.

In this paper, we take the first step towards systematical study of this phenomenon, by analyzing rigorously a single partial differential equation
modeling the fertilization process:
\begin{equation}\label{chemo}
\partial_t \rho+u\cdot\nabla\rho =
\Delta \rho+ \chi \nabla (\rho \nabla(\Delta)^{-1}\rho)  -  \epsilon \rho^q, \,\,\,\rho(x,0)=\rho_0(x),~~x\in\Rm^d.
\end{equation}
Here, in the simplest approximation, we consider just one density, $\rho(x,t) \geq 0,$
corresponding to the assumption that  the densities of sperm and egg gametes are identical.
The vector field $u$ in \eqref{chemo} models the ambient ocean flow, is
divergence free, regular and prescribed,
independent of $\rho.$
The second term on the right is the standard chemotactic term,
in the same form as it appears in the (simplified) Keller-Segel equation (see \cite{Pert}). This term describes
the tendency of $\rho(x,t)$ to move along the gradient of the chemical whose distribution is equal to $-\Delta^{-1}\rho.$
This is an approximation to the full Keller-Segel system based on the assumption of chemical diffusion being much faster than diffusion of gamete densities.
The term $(-\epsilon\rho^q)$ models the reaction (fertilization). The parameter $\epsilon$ regulates the strength of the fertilization process. The value
of $\epsilon$ is small due to the fact that an egg gets fertilized only if a sperm attaches to a certain limited area on its surface (estimated to be about
1\% of the total egg surface in, for example, sea urchins eggs \cite{Vogel}).
We do not account for the product of the reaction -- fertilized eggs --
which drop out of the process.
We are interested in  the behavior of
\[
m_0(t)=\int_{\Rm^d}\rho(x,t)dx,
\]
which is the total fraction of the unfertilized eggs by time $t$.
It is easy to see that $m_0(t)$ is monotone decreasing. High efficiency fertilization corresponds
to $m_0(t)$ becoming small with time, as almost all egg gametes are fertilized.
We prove the following results.
\begin{theorem}\label{chem1}
Let $\rho(x,t)$ solve \eqref{chemo} with a divergence free $u(x,t) \in C^\infty(\Rm^d \times [0,\infty))$ and initial data $\rho_0 \geq 0 \in {\mathcal S}(\Rm^d)$
(the Schwartz class).
Assume that $qd>d+2$, 
and the chemotaxis is absent: $\chi=0.$ Then there exists a constant $\mu_0$ depending only on
$\epsilon,$ $q,$ $d$ and $\rho_0(x)$ but not on $u(x,t)$ such that
$m_0(t)\ge \mu_0$ for all $t\ge 0$.

Moreover, $\mu_0 \rightarrow m_0(0)$ as $\epsilon \rightarrow 0$ while $\rho_0, u$ and $q$ are fixed.
\end{theorem}
\it Remarks. \rm 1. Observe that the constant $\mu_0$ does not depend on $u.$ No matter how strong the flow is or how it varies in time and space, it cannot
enhance the reaction rate beyond a certain definitive limit. \\
2. The condition $qd>d+2$ does not include the most natural case of $d=q=2.$ Dimension two corresponds to the surface of the ocean,
and $q=2$ corresponds to the product of egg and sperm densities. Our preliminary calculations show, however, that the
mathematics of $d=q=2$ case is different and more subtle. Then
the $L^1$ norm of $\rho$ for sufficiently rapidly decaying initial data
does go to zero but only very slowly in time.
The difference between chemotactic and chemotactic-free equation \eqref{chemo} in this case is likely to manifest itself
in the time scales of the fertilization process:
in the presence of chemotaxis the $L^1$ norm initially decays much faster. We will
address this issue in a separate publication, to keep the present
paper as transparent as possible. \\
3. The case of small $\epsilon$ is interesting due to experimental relevance. It will be clear from the proof that decrease of the $L^1$
norm, $m_0(0) - \lim_{t \rightarrow \infty}m_0(t)$ is of the order $\epsilon$ when chemotaxis is absent. \\
4. The condition that $\rho_0 \in {\mathcal S}$ can of course be weakened. What we need is the initial data that is decaying sufficiently quickly and is
minimally regular. Similarly, the condition that $u$ is smooth can be weakened to, say, $C^1$ without much difficulty. We decided not to pursue
most optimal regularity conditions on $\rho_0$ and $u$ in this paper to simplify presentation. \\
5. By $u \in C^\infty(\Rm \times [0,\infty))$ we mean that bounds on every derivative of $u$ are uniform over all $(x,t) \in \Rm^d \times [0,T],$ for every $T>0.$

On the other hand, in the presence of chemotaxis, we have
\begin{theorem}\label{chem2}
Let $\rho(x,t)$ solve \eqref{chemo} with a
divergence free $u(x,t) \in C^\infty(\Rm^d \times [0,\infty))$ and fixed initial data $\rho_0\geq 0 \in {\mathcal S}.$
Assume that $d=2,$ and $q$ is a positive integer greater than $2.$ Then we have that $m_0(t)\rightarrow c(\chi,\rho_0,u)>0$
as $t \rightarrow \infty,$
but $c(\chi,\rho_0,u)\to  0$
as $\chi \rightarrow \infty$, with $q$, $\rho_0$ and $u$ fixed.
\end{theorem}
\it Remarks. \rm
1. In general, chemotaxis can lead to finite time blow up of the solution (see \cite{Pert} for references).
However, it is known that the presence of the reaction can lead
to global regularity (see \cite{Wink1} for a slightly different model where references to some earlier works can also be found). We will
sketch the global regularity proof for solutions of the equation \eqref{chemo} in Appendix I.   \\
2. We prove more (see Theorem~\ref{chemore}). Here we stated the result in the simplest form to avoid technicalities.\\
3. An interesting feature of Theorem~\ref{chem2} is that $c(\chi,\rho_0,u) \rightarrow 0$ uniformly in reaction strength $\epsilon.$ The timescale
to achieve most of decay also does not depend on $\epsilon$ (Theorem~\ref{chemore}). The mechanism of this effect will be clear from the
proof. 
Roughly speaking, what happens is the solution tries to blow up to a $\delta$ function profile. The reaction term, however, prevents blow up from happening.
However, the solution gets larger if $\epsilon$ is small. Chemotaxis and reaction balance in a way to produce same order reaction effect independently of
$\epsilon.$ \\
4. The case $d>2$ is mathematically different and it is not clear that the $L^1$ norm may become arbitrarily small in this case even
with strong chemotaxis aid. There appears to be a genuine mathematical reason why coral gametes rise to the surface instead of trying to find each other in the three-dimensional ocean!

Hence our model implies that the chemotactic term, as opposed to the flow and diffusion alone, can
account for highly efficient fertilization rates that are observed
in nature. Moreover, Theorems~\ref{chem1} and \ref{chem2} suggest that the presence of chemotaxis may be a necessary and crucial aspect of the fertilization process.
Of course, a more realistic model of the process is a system of equations involving two different densities.
We will show that even for the system case, the flow
can only have a limited effect on fertilization efficiency, similarly to our simple model. It is possible that in the system case the flow and chemotaxis can play supplementary
role, with flow acting on larger and chemotaxis on smaller length scales. Note that one can expect that in the system setting the chemotaxis effect is weaker,
since it only appears in the equation for the sperm density.
The influence of chemotaxis in the system setting, and investigation
of quadratic reaction term are left for a later study.

\section{The reaction-advection-diffusion case}\label{rad}

In this section, we prove Theorem~\ref{chem1}. Consider equation \eqref{chemo} with $\chi=0:$
\begin{equation}\label{har}
\partial_t \rho +u\cdot\nabla\rho=
\Delta \rho   - \epsilon \rho^q, \,\,\,\rho(x,0)=\rho_0(x).
\end{equation}
As the first step, observe that by comparison principle, $\rho(x,t) \leq b(x,t),$ where
\begin{equation}\label{heat}
\partial_t b +u\cdot\nabla b= \Delta b, \,\,\,b(x,0)=\rho_0(x).
\end{equation}
Also, note that since $\rho(x,t) \geq 0,$
\[ \partial_t \|\rho(\cdot,t)\|_{L^1} = \partial_t \int_{\Rm^d} \rho(x,t)\,dx = - \epsilon \int_{\Rm^d} \rho^q(x,t)\,dx \geq - \epsilon \int_{\Rm^d} b^q(x,t)\,dx. \]
Therefore, the behavior of the $L^q$ norm of $b$ can be used for estimating decay of the $L^1$ norm of $\rho.$
We have the following lemma, similar in spirit (and proof)
to~Lemma 3.1 of~\cite{FKR}.
\begin{lemma}\label{nash}
There exists $C=C(d)$ that, in particular, does not depend on the flow $u$,
such that
\begin{equation}\label{decb}
\|b(\cdot,t)\|_{L^2} \leq {\rm min}(\|b_0\|_{L^2},Ct^{-d/4}\|b_0\|_{L^1}), \,\,\, \|b(\cdot ,t)\|_{L^\infty} \leq {\rm min}(\|b_0(x)\|_{L^\infty},Ct^{-d/2}\|b_0\|_{L^1}).
\end{equation}
\end{lemma}
\begin{proof}
By Nash inequality \cite{Nash}, we have
\[ \|b\|_{L^2}^{1+\frac2d} \leq C(d)\|b\|_{L^1}^{2/d} \|\nabla b\|_{L^2}. \]
Multiplying \eqref{heat} by $b,$ integrating,
and using incompressibility of $u,$ we get
\[ \frac12 \partial_t \|b\|_{L^2}^2 = -\|\nabla b\|_{L^2}^2  \leq -C\frac{\|b\|_{L^2}^{2+\frac{4}{d}}}{\|b\|_{L^1}^{\frac{4}{d}}}=
-C\frac{\|b\|_{L^2}^{2+\frac{4}{d}}}{\|b_0\|_{L^1}^{\frac{4}{d}}}. \]
We used the conservation of the $L^1$-norm of $b$ in the last step. Set $z(t) = \|b(\cdot,t)\|_{L^2}^2.$ Then
\[ z'(t) \leq - Cz(t)^{1+\frac{2}{d}}\|b_0\|_{L^1}^{-\frac{4}{d}}. \]
Solving this differential inequality, we get
\[ z(t) \leq \left(\frac{2Ct}{d\|b_0\|_{L^1}^{4/d}}+\frac{1}{\|b_0\|_{L^2}^{4/d}}\right)^{-d/2}, \]
implying
\[ \|b(\cdot, t)\|_{L^2}^2 \leq {\rm min}\left(\|b_0\|^2_{L^2}, C(d)t^{-d/2}\|b_0\|^2_{L^1} \right). \]
This gives the first inequality in \eqref{decb}.

The second inequality in \eqref{decb} follows from a simple duality argument using incompressibility of $u.$
Indeed, consider $\theta(x,s),$ a solution of
\[
\partial_s \theta +u(x,t-s)\cdot \nabla \theta =
\Delta \theta, \,\,\, \theta(x,0)=\theta_0(x) \in {\mathcal S}.
\]
A direct calculation shows that
\[ \frac{d}{ds} \int_{\Rm^d} b(x,s) \theta(x,t-s)\,dx =0. \]
When $s=t,$ we get
\[ \left| \int_{\Rm^d} b(x,t) \theta_0(x)\,dx \right| \leq \|b(x,t)\|_{L^2} \|\theta_0\|_{L^2} \leq C(d)t^{-d/4}\|b_0\|_{L^1}\|\theta_0\|_{L^2}. \]
For $s=0,$ this implies
\[ \left| \int_{\Rm^d} b_0(x) \theta(x,t)\,dx \right| \leq C(d)t^{-d/4}\|b_0\|_{L^1}\|\theta_0\|_{L^2} \]
for every $b_0,\theta_0 \in {\mathcal S}.$ Hence
\begin{equation}\label{thcon14}
\|\theta(x,t)\|_{L^\infty} \leq C(d)t^{-d/4}\|\theta_0\|_{L^2}
\end{equation}
 for every $\theta_0 \in L^2.$
To finish the proof of the Lemma, given $t>0,$ note that
\[ \|b(x,t)\|_{L^\infty} \leq C(d)(t/2)^{-d/4} \|b(x,t/2)\|_{L^2} \leq C(d)t^{-d/2}\|b_0\|_{L^1}. \]
Here in the second step we used \eqref{thcon14} and adjusted $C(d).$

\end{proof}

For a more precise estimate on the residual mass
$\mu_0,$ we need one more lemma.
\begin{lemma}\label{normcomp}
Assume that $\rho(x,t)$ solves \eqref{har} with a
smooth bounded incompressible $u$ and $\rho_0 \in {\mathcal S}.$ Then for every $t>0$
we have
\[ \frac{\|\rho(x,t)\|_{L^p}}{\|\rho(x,t)\|_{L^1}} \leq \frac{\|\rho_0\|_{L^p}}{\|\rho_0\|_{L^1}}\,\,\,
{\rm for}\,\,\,{\rm all}\,\,\,1 \leq p \leq \infty.\]
\end{lemma}
\begin{proof}
For $p=1$ the result is immediate.
Consider some $1<p<\infty,$ and look at
\begin{eqnarray*}
&&\frac{\partial}{\partial t}
\left(\frac{\int_{\Rm^d} \rho^p\,dx}{\left(\int_{\Rm^d}\rho\,dx\right)^p}\right) =
p\left(\int_{\Rm^d} \rho\,dx \right)^{-p-1}\\
&&\times\left[
{\int_{\Rm^d}
\rho^{p-1}(-u \cdot \nabla\rho + \Delta \rho - \epsilon \rho^q)\,dx \int_{\Rm^d} \rho\,dx
- \int_{\Rm^d}
\rho^p\,dx \int_{\Rm^d} (-u\cdot \nabla\rho +\Delta \rho -\epsilon \rho^q )\,dx}\right]
\end{eqnarray*}
Consider the term in the second line above,
which after integration by parts simplifies to
\[
\left( -(p-1)\int_{\Rm^d} \rho^{p-2}|\nabla\rho|^2\,dx -
\epsilon \int_{\Rm^d} \rho^{q+p-1}\,dx \right)\int_{\Rm^d}\rho\,dx+
\epsilon \int_{\Rm^d}\rho^p\,dx \int_{\Rm^d}\rho^q\,dx.
\]
This does not exceed
\[ -\epsilon \int_{\Rm^d} \rho^{q+p-1}\,dx \int_{\Rm^d}\rho\,dx+
\epsilon \int_{\Rm^d}\rho^p\,dx \int_{\Rm^d}\rho^q\,dx, \]
which is less than or equal to zero by an
application of H\"older's inequality.

The $p=\infty$ case follows by a limiting procedure since $\rho(x,t) \in {\mathcal S}$
for all $t$.
\end{proof}
We are ready to prove Theorem~\ref{chem1}.
\begin{proof}[Proof of Theorem~\ref{chem1}]
The idea of the proof is very simple. We will show that if $L^1$-norm of $\rho$
at some time $t_0$ is sufficiently small then for all times $t>t_0$ the $L^1$-norm
of $\rho(x,t)$ can not drop below $\|\rho(t_0)\|_{L^1}/2$. This shows that
$\rho(x,t)$ can not tend to zero as $t\to +\infty$.

Recall that for every $t,$
\[ \partial_t \int_{\Rm^d} \rho(x,t)\,dx = -\epsilon\int_{\Rm^d} \rho(x,t)^q \,dx \geq -\epsilon\int_{\Rm^d} b(x,t)^q\,dx, \]
where $b$ is given by \eqref{heat}. By Lemma~\ref{nash} and H\"older's inequality,
\[ \int_{\Rm^d} b(x,t)^q \,dx \leq C {\rm min} \left(\|\rho_0\|_{L^q}^q,t^{-\frac{d(q-1)}{2}}\|\rho_0\|^q_{L^1}\right). \]
Thus, for every $\tau>0,$
\begin{eqnarray}\nonumber
\int_{t_0}^\infty dt \int_{\Rm^d} b(x,t)^q \,dx \leq C(d)\left( \|\rho(\cdot,t_0)\|_{L^q}^q \tau + \|\rho(\cdot,t_0)\|_{L^1}^q\int_{t_0+\tau}^\infty (t-t_0)^{-\frac{d(q-1)}{2}}\,dt \right)\\
\label{keyb11}\leq C(d,q)\left( \|\rho(\cdot,t_0)\|_{L^\infty}^{q-1}\|\rho(\cdot,t_0)\|_{L^1}\tau + \|\rho(\cdot, t_0)\|_{L^1}^q \tau^{\frac{d+2-qd}{2}}\right).
\end{eqnarray}
We used the assumption $qd>d+2$ when evaluating integral in time.

Assume, on the contrary, that the $L^1$ norm of $\rho$ does go to zero for some $u.$
Consider some time $t_0>0$ when $\|\rho(\cdot, t_0)\|_{L^1}$ is sufficiently small (we'll have a precise bound later). Using Lemma~\ref{normcomp} and \eqref{keyb11},
we see that further decrease of the $L^1$ norm
from that level is bounded by
\begin{equation}\label{kb2}
\|\rho(\cdot,t_0)\|_{L^1} - \|\rho(\cdot,t)\|_{L^1}
\le C(d,q) \epsilon \left( \frac{\|\rho_0\|^{q-1}_{L^\infty}}{\|\rho_0\|^{q-1}_{L^1}}\|\rho(\cdot,t_0)\|_{L^1}^q\tau + \|\rho(\cdot,t_0)\|_{L^1}^q \tau^{\frac{d+2-qd}{2}}\right),
\end{equation}
for all $t>t_0,$ $\tau>0$.
Choosing $\tau$ to minimize the expression \eqref{kb2}, we find that for every $t>t_0,$
\begin{equation}\label{l1con2}
\|\rho(\cdot,t_0)\|_{L^1} - \|\rho(\cdot,t)\|_{L^1} \leq C(q,d)\epsilon \|\rho(\cdot,t_0)\|_{L^1}^q \left(\frac{\|\rho_0\|_{L^\infty}}{\|\rho_0\|_{L^1}}\right)^{\frac{qd-d-2}{d}}.
\end{equation}
If $\|\rho(\cdot, t)\|_{L^1}\to 0$ as $t\to +\infty$, we may choose $t_0$ so that
\begin{equation}\label{rhoepscon} C(q,d) \epsilon \|\rho(\cdot,t_0)\|_{L^1}^{q-1} \left(\frac{\|\rho_0\|_{L^\infty}}{\|\rho_0\|_{L^1}}\right)^{\frac{qd-d-2}{d}} \leq \frac12. \end{equation}
Then we get that
\[ \|\rho(\cdot,t)\|_{L^1} \geq \frac12 \|\rho(\cdot,t_0)\|_{L^1} \geq \mu_0(q,d,\rho_0) \equiv {\rm min}\left(\frac12 \|\rho_0\|_{L^1},\frac{1}{2^{\frac{q}{q-1}}\epsilon^{\frac{1}{q-1}} C(q,d)^{\frac{1}{q-1}}}
\left( \frac{\|\rho_0\|_{L^1}}{\|\rho_0\|_{L^\infty}}\right)^{1-\frac{2}{d(q-1)}}\right)\]
for every $t>t_0.$ This is a contradiction to the assumption that
$\|\rho(t)\|_{L^1}\to 0$ as $t\to+\infty$.
This argument can also be used to define $\mu_0$ in the statement of the theorem.
The last statement of the theorem is easy to prove by changing the condition $\leq \frac12$ in \eqref{rhoepscon} to $\leq \kappa$ where $\kappa$ can be taken as small
as desired.
\end{proof}

\section{The reaction-advection-diffusion case: a system}\label{radsys}

In this section we show that the results of Section~\ref{rad} largely extend to a more general model.
Consider the following system
\begin{eqnarray}\label{seq}
\partial_t s = (u \cdot \nabla)s +\kappa_1\Delta s - \epsilon (se)^{q/2},\,\,\,s(x,0)=s_0(x) \\
\label{eeq}
\partial_t e = (u \cdot \nabla)e +\kappa_2\Delta e - \epsilon (se)^{q/2}, \,\,\,e(x,0)=e_0(x).
\end{eqnarray}
Here $s(x,t)$ and $e(x,t)$ are sperm and egg densities respectively. The following analog of Theorem~\ref{chem1} holds.
\begin{theorem}\label{chem1sys}
Let $s(x,t),e(x,t)$ solve \eqref{seq},\eqref{eeq} with divergence free $u(x,t) \in C^\infty(\Rm^d \times [0,\infty))$ and initial data $s_0,e_0 \in S$.
Assume that $qd>d+2,$ $q>2$ and the chemotaxis is absent: $\chi=0.$ Then there exists a constant $\mu_1$ depending only on
$\epsilon,$ $q,$ $d$ and $e_0(x),s_0(x)$ such that
the $L^1$ norms of $s(x,t)$ and $e(x,t)$ remain greater than $\mu_1$ for all times.
\end{theorem}
\it Remarks. \rm 1. The condition $q>2$ can be omitted if $\|s_0\|_{L^1}=\|e_0\|_{L^1}$. \\
2. Similarly to Theorem~\ref{chem1}, one can show that $\lim_{t \rightarrow \infty} \|s(\cdot,t)\|_{L^1} \stackrel{\epsilon \rightarrow 0}{\longrightarrow} \|s_0\|_{L^1}$
and $\lim_{t \rightarrow \infty} \|e(\cdot,t)\|_{L^1} \stackrel{\epsilon \rightarrow 0}{\longrightarrow} \|e_0\|_{L^1}$ provided that $q,$ $u,$ $s_0$ and $e_0$ remain
fixed.
\begin{proof}
As before, we know that $s(x,t) \leq \overline{s}(x,t)$ and $e(x,t) \leq \overline{e}(x,t)$ where $\overline{s},\overline{e}$ solve
\eqref{heat} with initial data $s_0$ and $e_0$,  and the diffusion coefficients $\kappa_1$ and $\kappa_2$, respectively.
Lemma~\ref{nash} can still be used to control $\overline{s},\overline{e}.$ Instead of Lemma~\ref{normcomp}, we will use a cruder bound.

Observe that if $\|e_0\|_{L^1} \ne \|s_0\|_{L^1},$ then
the $L^1$ norm that is larger initially remains larger
than the other norm. Hence, assume without loss of generality
that $\|e_0\|_{L^1} \leq \|s_0\|_{L^1}$ and focus
on the decay of $\|e(\cdot,t)\|_{L^1}.$ Let us estimate the decay after some
time $t_0:$
\begin{eqnarray}\nonumber
&&\left| \int_{t_0}^\infty \, dt \int_{\Rm^d} s(x,t)^{q/2}e(x,t)^{q/2}\,dx \right|
\leq \\ \nonumber
&&\left| \int_{t_0}^{t_0+\tau} \, dt \int_{\Rm^d} s(x,t)^{q/2}e(x,t)^{q/2}\,dx\right|
+\left|\int_{t_0+\tau}^\infty \, dt \int_{\Rm^d} s(x,t)^{q/2}e(x,t)^{q/2}\,dx \right|\\ \nonumber
&&\leq
\tau \|s(\cdot,t_0)\|_{L^\infty}^{q/2} \|e(\cdot,t_0)\|_{L^\infty}^{\frac{q}{2}-1}\|e(\cdot,t_0)\|_{L^1} +\int_{t_0+\tau}^{\infty}
 \|s(\cdot,t)\|_{L^q}^{q/2}\|e(\cdot,t)\|_{L^q}^{q/2}\,dt \\
 &&\leq \label{tausys2}
\tau \|s_0\|_{L^\infty}^{q/2}\|e_0\|_{L^\infty}^{\frac{q}{2}-1}\|e(\cdot,t_0)\|_{L^1}+C\tau^{1-\frac{d(q-1)}{2}}\|s_0\|_{L^1}^{q/2}\|e(\cdot,t_0)\|_{L^1}^{q/2}.
\end{eqnarray}
Choosing $\tau$ to minimize \eqref{tausys2} leads to
\begin{equation}\label{finsys45}
 \|e(\cdot,t_0)\|_{L^1} - \|e(\cdot,t)\|_{L^1}
 \leq C(q,d)\epsilon \|s_0\|_{L^\infty}^{\frac{q(qd-d-2)}{2d(q-1)}}\|s_0\|_{L^1}^{\frac{q}{d(q-1)}}\|e_0\|_{L^\infty}^{\frac{(q-2)(qd-d-2)}{2d(q-1)}}\|e(\cdot,t_0)\|_{L^1}^{1+\frac{q-2}{d(q-1)}}.
 \end{equation}
Suppose that $\|e(\cdot,t)\|_{L^1}$ does go to zero as $t \rightarrow \infty.$ Choose $t_0$ so that $C\|e(x,t_0)\|_{L^1}^{\frac{q-2}{d(q-1)}} <\frac12$ (where $C$ is
the constant in front of $\|e(\cdot,t_0)\|_{L^1}^{1+\frac{q-2}{d(q-1)}}$ in \eqref{finsys45}).
In this case, due to \eqref{finsys45}, the $L^1$ of $e(x,t)$ can never drop below half of its value at $t_0.$
This is a contradiction.
\end{proof}

\section{Reaction enhancement by chemotaxis}\label{chemotaxis}

In this section, we will show that chemotaxis, as opposed to
a divergence free fluid flow, can, in principle,
make reaction as efficient as needed.
We consider the equation
\begin{equation}\label{chemnew}
\partial_t \rho =  \Delta \rho-(u \cdot \nabla)\rho + \chi \nabla (\rho \nabla(\Delta)^{-1}\rho)  - \epsilon \rho^q, \,\,\,\rho(x,0)=\rho_0(x).
\end{equation}
We will prove that the large time limit of the $L^1$ norm of $\rho(x,t)$
goes to zero as chemotaxis coupling increases, independently of $\epsilon.$ On the other hand, we will also prove
lower bounds showing that the $L^1$ norm does not go to zero as $t \rightarrow \infty$ for each fixed coupling. Before we state the main results of this section,
there is an auxiliary issue we need to settle.
In general, solutions to the chemotaxis equation may lose regularity in
a finite time (see e.g. \cite{Pert} for further references). As Theorem~\ref{globex} below shows, this does not happen with the additional negative reaction term $-\epsilon \rho^q,$ $q>2$
in the right hand side: solutions with smooth initial data stay smooth. We will work with initial data which is
concentrated in a finite region, in particular, with a finite second moment.
As we will see, this property is also preserved by the evolution. Let us define
\[ \|f\|_{M_n} = \int_{\Rm^d}(|\nabla f|+ |f(x)|)(1+|x|^n)\,dx. \]
Let $H^s$ denote the standard Sobolev spaces in $\Rm^d.$
Define a Banach space $K_{s,n}$ with the norm $\|f\|_{K_{s,n}} = \|f\|_{H^s}+ \| f\|_{M_n}.$ Then we have
\begin{theorem}\label{globex}
Assume that $q>2,$ $n>0$ and $s>d/2+1$ are integers and $\rho_0 \in K_{s,n}.$ Suppose that $u \in C^\infty(\Rm^d \times [0,\infty))$ is
divergence free. Then there exists a unique solution $\rho(x,t)$ of the equation
\eqref{chemnew} in $C(K_{s,n}, [0,\infty)) \cap C^\infty(\Rm^d \times (0,\infty)).$
\end{theorem}
\noindent
The proof of Theorem~\ref{globex} uses fairly standard techniques; we sketch it in Appendix I.

First, we prove the bound showing reaction enhancement by chemotaxis. Let us define
\[
m_2 = {\rm min}_{x_0} \int_{\Rm^d}|x-x_0|^2 \rho_0(x)\,dx.
\]
\begin{theorem}\label{chemore}
Let $d=2,$ and suppose that $u \in C^\infty(\Rm^d \times [0,\infty))$ is
divergence free. Assume that $q>2$, $s>d/2+1$ and $n \geq 2$ are integers and $\rho(x,t)$ solves \eqref{chemnew} with $\rho_0 \geq 0 \in K_{s,n}.$
Then

a. If $u=0,$ then $\lim_{t \rightarrow \infty}\|\rho(\cdot,t)\|_{L^1} \leq 2\chi^{-1}.$
More precisely, for every $\tau>0,$ we have
\begin{equation}\label{dec1a}
\|\rho(\cdot,\tau)\|_{L^1} \leq \frac{2}{\chi}\left(1+\sqrt{1+\frac{\chi m_2}{4\tau}}\right).
\end{equation}

b. If $u \ne 0,$ then $\lim_{t \rightarrow \infty}\|\rho(\cdot,t)\|_{L^1} \leq C(u,m_2)\chi^{-2/3}.$ Moreover,
for $0 \leq \tau \leq \chi^{1/3}$ we have
\begin{equation}\label{dec1b}
\|\rho(\cdot,\tau)\|_{L^1} \leq C(u,m_2)(\chi\tau)^{-1/2}.
\end{equation}
\end{theorem}
\it Remark. \rm Note, in particular, that if $u=0,$ the level $\|\rho(\cdot,\tau)\|_{L^1} \sim \chi^{-1}$ will be reached in
at most $\tau \sim \chi,$
while the level $\sim \chi^{-1/2}$ in at most $\tau \sim 1.$ If $u \ne 0,$ the upper bound on the time scale to reach the $L^1$ norm level $\sim \chi^{-1/2}$  is also $\tau\sim 1$.
\begin{proof}
Since $\rho_0 \in K_{s,n},$ there exists $x_0$ such that $\int_{\Rm^2}|x-x_0|^2 \rho_0(x)\,dx =m_2.$ Set $x_0=0$ for simplicity. Consider
\begin{equation}\label{cc1}
\partial_t \int_{\Rm^2} |x|^2 \rho\,dx = \int_{\Rm^2}|x|^2 (u \cdot \nabla)\rho\,dx +\int_{\Rm^2}|x|^2 \Delta \rho \,dx+\chi \int_{\Rm^2}
|x|^2 \nabla(\rho \nabla \Delta^{-1}\rho)\,dx - \epsilon \int_{\Rm^2}|x|^2\rho^q\,dx.
\end{equation}
Observe that due to $\nabla \cdot u =0,$
\[
\int_{\Rm^2} |x|^2(u \cdot \nabla)\rho \,dx = -2 \int_{\Rm^2} (x \cdot u) \rho\,dx,
\]
and in dimension two
\[
\int_{\Rm^2}|x|^2 \Delta \rho \,dx =4\int_{\Rm^2}\rho\,dx.
\]
For the chemotaxis term, we have
\[ \int_{\Rm^2} |x|^2 \nabla(\rho \nabla \Delta^{-1}\rho)\,dx =
-2 \int_{\Rm^2\times\Rm^2}  \frac{x\cdot(x-y)}{|x-y|^2} \rho(x,t) \rho(y,t)\,dxdy =
-\left( \int_{\Rm^2} \rho \,dx \right)^2. \]
In the last step, we used symmetrization in $x,y.$ Due to Theorem~\ref{globex}, all integrations by parts are justified for all $t \geq 0.$
Therefore, we can recast \eqref{cc1} as
\begin{equation}\label{cc2}
\partial_t \int_{\Rm^2} |x|^2 \rho\,dx = -2 \int_{\Rm^2} (x \cdot u) \rho\,dx +
4\int_{\Rm^2}\rho\,dx - \chi \left( \int_{\Rm^2} \rho \,dx \right)^2
- \epsilon \int_{\Rm^2}|x|^2\rho^q\,dx.
\end{equation}
First let us set $u=0$ in \eqref{cc2}. Suppose that $\|\rho(\cdot,t)\|_{L^1} \geq Y$ for all $t \in [0,\tau]$, and $Y\ge 4/\chi$.
It follows from \eqref{cc2}  that we need
$\tau Y(\chi Y - 4) \le m_2$ to avoid contradiction that $m_2$ vanishes. This quadratic inequality translates into \eqref{dec1a}.

Now, assume that $u$ is an arbitrary smooth divergence free vector field. In this case, we further estimate
\[ \left| \int_{\Rm^2} x \cdot u \rho\,dx \right| \leq \|u\|_{L^\infty}^2 \chi^\beta \int_{\Rm^2} \rho\,dx +
\chi^{-\beta} \int_{\Rm^2} |x|^2 \rho\,dx,
\]
with $\beta>0$ to be chosen.
Then,  it follows  from \eqref{cc2} that
\[ \partial_t \int_{\Rm^2} |x|^2 \rho\,dx <  2\chi^{-\beta}\int_{\Rm^2} |x|^2 \rho\,dx +
\left(4+2\chi^\beta \|u\|_{L^\infty}^2 -\chi \int_{\Rm^2}\rho\,dx \right) \int_{\Rm^2} \rho \,dx, \]
and thus
\begin{equation}\label{cc4}
\partial_t \left( e^{-2\chi^{-\beta}t} \int_{\Rm^2} |x|^2 \rho\,dx \right) < e^{-2\chi^{-\beta}t}
  \left(4+2\chi^\beta \|u\|_{L^\infty}^2 -\chi \int_{\Rm^2}\rho\,dx \right) \int_{\Rm^2} \rho \,dx  .
\end{equation}
Assume now that for all $t \in [0,\tau],$ we have $\|\rho(\cdot,t)\|_{L^1} \geq Y>0$, and that
\[
Y \geq \frac{2}{\chi}(2+\chi^\beta \|u\|_{L^\infty}^2).
\]
Then, the integral in time of the right hand side in \eqref{cc4}
over $[0,\tau]$ can be estimated from above by
\begin{equation}\label{cc5} \int_0^\tau e^{-2\chi^{-\beta}t} Y
 \left(4+2\chi^\beta \|u\|_{L^\infty}^2 -\chi Y \right)\,dt = \left( 1- e^{-2\chi^{-\beta}\tau} \right) Y\chi^\beta
 \left(2+\chi^\beta \|u\|_{L^\infty}^2 -\chi Y/2 \right). \end{equation}
Setting $\tau = \chi^{\beta}$,  we see that to avoid a contradiction, we need
\begin{equation}\label{cc6}
(1-e^{-2}) \chi^\beta Y \left(\chi Y -4 -2 \chi^\beta \|u\|_{L^\infty}^2 \right) \leq 2 m_2.
\end{equation}
An elementary computation shows that the optimal choice that makes $Y$ the smallest is $\beta=1/3$.
Solving this quadratic inequality, we find that $\|\rho(\cdot,\tau=\chi^{1/3})\|_{L^1}$ cannot exceed $c(u,m_2)\chi^{-2/3}$.
More generally, for $0 < \tau \leq \chi^{1/3},$ we get from (\ref{cc5}) the bound
\[
\|\rho(\cdot,\tau)\|_{L^1} \leq C(u,m_2)(\chi\tau)^{-\frac12}.
\]
\end{proof}
Observe that the reaction term did not play any quantitative role in the estimates. In particular, all estimates of the $L^1$ norm decrease and
timescales are independent of $\epsilon.$ The only role the reaction term plays is making sure we have a smooth decaying solution, so that
all integrations by parts are justified. On a qualitative level, what happens is chemotaxis without reaction would lead to a $\delta$ function
profile blow up. With reaction present, the growth in the $L^\infty$ norm of the solution is determined by the balance between chemotaxis
and reaction term. Weaker fertilization coupling parameter $\epsilon$ leads to stronger aggregation due to chemotaxis
and thus effectively the same fertilization rates. In particular, the amount of the density that does react satisfies the lower bound independent
of $\epsilon.$ The same holds true for the time scale on which this reaction takes place - it is bounded from above by the pure chemotaxis blow
up time independently of $\epsilon.$

Next we prove a result in the opposite direction, showing that at least some estimates of Theorem~\ref{chemore} scale sharply in $\chi.$
\begin{theorem}\label{chemdiss}
Let $d = 2,$ and suppose that $u \in C^\infty(\Rm^d \times [0,\infty))$ is
divergence free. Assume that $q>2$, $s>d/2+1$
and $n \geq 2$ are integers and $\rho(x,t)$ solves \eqref{chemnew} with $\rho_0 \geq 0 \in K_{s,n}.$
Then $\lim_{t \rightarrow \infty}\|\rho(\cdot,t)\|_{L^1}>0.$ Moreover, for some initial data $\rho_0,$ $\|\rho(\cdot,t)\|_{L^1}$ remains above
$c(q,\rho_0)\chi^{-1}$ for all times.
\end{theorem}
\begin{proof}
Recall that $\partial_t \int_{\Rm^2} \rho(x,t)\,dx = - \int_{\Rm^2} \rho(x,t)^q\,dx.$ Let us derive estimates on $\|\rho\|_{L^q}.$
Multiplying \eqref{chemnew} by $\rho^{q-1}$ and integrating, we obtain
\begin{equation}\label{eq1} \frac1q \partial_t \int_{\Rm^2} \rho^q \, dx = \int_{\Rm^2} \rho^{q-1} \Delta \rho\,dx + \chi \int_{\Rm^2} \rho^{q-1} \nabla \cdot
(\rho \nabla \Delta^{-1}\rho)\,dx - \epsilon \int_{\Rm^2} \rho^{2q-1}\,dx. \end{equation}
Observe that
\[ \int_{\Rm^2} \rho^{q-1} \nabla \cdot
(\rho \nabla \Delta^{-1}\rho)\,dx = -(q-1)\int_{\Rm^2} \rho^{q-1} \nabla \rho \cdot \nabla \Delta^{-1}\rho\,dx = \frac{q-1}{q} \int_{\Rm^2}
\rho^{q+1}\,dx. \]
The last equality is obtained by integration by parts. Thus, we can rewrite \eqref{eq1} as
\begin{equation}\label{eq2}
\partial_t \int_{\Rm^2} \rho^q \, dx = -\frac{4(q-1)}{q} \int_{\Rm^2} | \nabla \rho^{q/2}|^2\,dx + \chi(q-1) \int_{\Rm^2} \rho^{q+1}\,dx
- q \epsilon \int_{\Rm^2} \rho^{2q-1}\,dx.
\end{equation}
Let us introduce $v = \rho^{q/2},$ and recall a Gagliardo-Nirenberg type inequality
\begin{equation}\label{nonstgn} \|v\|_{L^{2+\alpha}} \leq C(d,\alpha) \|\nabla v\|_{L^2}^{\frac{2}{2+\alpha}}
\|v\|_{L^{\frac{\alpha d}{2}}}^{\frac{\alpha}{2+\alpha}}, \end{equation}
which is valid for all $\alpha>0,$ $d \geq 1.$  In our case, we set $\alpha = 2/q$, and inequality \eqref{nonstgn} translates into
\[ \int_{\Rm^2} \rho^{q+1}\,dx \leq C(q) \int_{\Rm^2} |\nabla \rho^{q/2}|^2\,dx \int_{\Rm^2} \rho\,dx. \]
Observe that $\alpha d/2 <1.$ While inequalities of this kind are well known to the experts \cite{Maz}, the references that include the case of exponents less than one
are not common. For the sake of completeness, we provide a sketch of a simple proof of inequality \eqref{nonstgn} in Appendix II.
Therefore, from \eqref{eq2} we can conclude that
\[ \partial_t \int_{\Rm^2} \rho^q \,dx \leq -\frac{q-1}{q} \int_{\Rm^2} |\nabla \rho^{q/2}|^2 \,dx \left( 4 - C(q)\chi \int_{\Rm^2} \rho \,dx \right) -
q \epsilon \int_{\Rm^2} \rho^{2q-1}\,dx. \]
Now, suppose that $C(q)\chi\int_{\Rm^2} \rho(x,t)\,dx$ drops below $2$ at some time $t_0.$ Then, for all later times, we get
\begin{equation}\label{eq3}
\partial_t \int_{\Rm^2} \rho^q \,dx \leq -C(q) \int_{\Rm^2} |\nabla \rho^{q/2}|^2 \,dx
\end{equation}
(we use $C(q)$ for a positive constant depending only on $q$ that may change from line to line).
Let us recall another Gagliardo-Nirenberg inequality
\begin{equation}\label{nonstgn1} \|v\|_{L^2}^{1+\frac{2}{d(q-1)}} \leq C(q,d) \|\nabla v\|_{L^2} \|v\|_{L^{2/q}}^{\frac{2}{d(q-1)}}. \end{equation}
Applying it in \eqref{eq3} with $v =\rho^{q/2}$ in $d=2$ leads to
\[
\partial_t \int_{\Rm^2} \rho^q \, dx \leq -C(q)
\left( \int_{\Rm^2} \rho^q\,dx \right)^{1+\frac{1}{q-1}}\left(\int_{\Rm^2} \rho\,dx \right)^{-\frac{q}{q-1}}.
\]
Solving this differential inequality, and using the fact that $\int_{\Rm^2}\rho\,dx$ is monotone decreasing, leads to
\[ \int_{\Rm^2}\rho(x,t)^q\,dx \leq {\rm min}\left( \int_{\Rm^2} \rho(x,t_0)^q\,dx, C(q)(t-t_0)^{-q+1} \left( \int_{\Rm^2}\rho(x,t_0)\,dx \right)^q \right).  \]
Then the argument identical to that in the proof of Theorem~\ref{chem1} implies that
\begin{equation}\label{eq5}
 {\rm inf}_t \int_{\Rm^2} \rho(x,t)\,dx \geq {\rm min} \left( \frac12 \|\rho(\cdot, t_0)\|_{L^1}, C(q) \epsilon^{-\frac{1}{q-1}} \left( \frac{\|\rho(\cdot,t_0)\|_{L^1}}{\|\rho(\cdot,t_0)\|_{L^\infty}} \right)^{\frac{q-2}{q-1}} \right)
\end{equation}
(observe that the proof of Lemma~\ref{normcomp} goes through when $C(q)\chi\int_{\Rm^2} \rho(x,t)\,dx<2$).
Since we have a uniform in time upper bound $\|\rho(x,t)\|_{L^\infty} \leq {\rm max}((\chi/\epsilon)^{\frac{1}{q-2}},\|\rho_0\|_{L^\infty})$ (see Lemma~\ref{linfty} below), \eqref{eq5} implies the first statement of the theorem. Moreover, we can always take initial data
such that $t_0=0,$ and the $L^\infty$ norm of $\rho_0$ is sufficiently small ($\leq (\chi/\epsilon)^{\frac{1}{q-2}}$), making the bound on the right hand side of \eqref{eq5} equal to $c(q)\chi^{-1}.$
This proves the second statement of the theorem.
\end{proof}

\section{Appendix I: Global existence of smooth solutions}\label{app1}

Here we prove Theorem~\ref{globex}. We point out that equations involving chemotaxis and logistics or bistable type reactions have been
considered by many authors (see e.g. \cite{MiTsu,Osaki,TeWi,Wink1,Wink2} where further references can be found). In particular,
in \cite{TeWi,Wink2}, global regularity of solutions to a system similar to \eqref{chemo} was obtained. The main difference between these
works and what we need here is that \cite{TeWi,Wink2} work on a bounded domain. Since we consider a different setting and need different
control of the solution (including moments) we present a brief sketch of the proof of global regularity result in this appendix.

We begin with the construction of a
local solution in an appropriate space. We will consider arbitrary dimension $d.$
Recall that
\[
\|f\|_{M_n} = \int_{\Rm^d} (|\rho(x)| + |\nabla \rho(x)|)(1+|x|^n)\,dx,
\]
and the Banach space $K_{s,n}$ is defined by the norm
$\|f\|_{K_{s,n}} = \|f\|_{M_n} + \|f\|_{H^s}.$ First, we need a simple lemma
on the heat semigroup action in this space.

\begin{lemma}\label{elem1}
Assume that $\rho_0 \in K_{s,n},$ with $s\geq 0,n\geq 0.$ Then we have
\begin{equation}\label{mom1}
\|e^{t\Delta } \rho_0\|_{M_n} \leq C(1+t^{n/2})\|\rho_0\|_{M_n}, \,\,\, \|\nabla e^{t\Delta } \rho_0\|_{M_n} \leq C(t^{-1/2}+t^{(n-1)/2})\|\rho_0\|_{M_n};
\end{equation}
\begin{equation}\label{sob1}
\|e^{t\Delta } \rho_0\|_{H^s} \leq \|\rho_0\|_{H^s}, \,\,\,
\|\nabla e^{t\Delta } \rho_0\|_{H^s} \leq Ct^{-1/2} \|\rho_0\|_{H^s}.
\end{equation}
As a consequence,
\begin{equation}\label{kns1}
\|e^{t\Delta } \rho_0\|_{K_{s,n}} \leq C(1+t^{n/2})\|\rho_0\|_{K_{s,n}}, \,\,\,
\|\nabla e^{t\Delta } \rho_0\|_{K_{s,n}} \leq C(t^{-1/2}+t^{(n-1)/2})\|\rho_0\|_{K_{s,n}}.
\end{equation}
\end{lemma}
The proof of Lemma~\ref{elem1} is elementary and we omit it.

Next, we set up the contraction mapping argument for local existence. We will use the Banach space $X^T_{s,n} \equiv C(K_{s,n},[0,T])$ with a sufficiently small $T>0.$
Let us rewrite the equation \eqref{chemnew} in an integral form using the Duhamel principle.
\begin{equation}\label{duhamel}
\rho(x,t) = e^{t\Delta } \rho_0(x) + \int_0^t e^{  (t-s)\Delta} \left( -\nabla \cdot (u\rho) - \epsilon \rho^q +
\chi\nabla\cdot (\rho \nabla \Delta^{-1}\rho)\right)\,ds.
\end{equation}
Let us denote
\[ B_t (\rho) \equiv \int_0^t e^{  (t-s)\Delta} \left( -\nabla \cdot (u\rho) -
\epsilon \rho^q + \chi \nabla\cdot (\rho \nabla \Delta^{-1}\rho)\right)\,ds.
\]
We need the following auxiliary estimates.
\begin{lemma}\label{Blem}
Assume that $q,s,n$ are positive integers and $s>\frac{d}{2}+1.$ Let $f,g \in K_{s,n}.$ Then
\begin{eqnarray}\label{powsob}
\|f^q - g^q\|_{H^s} \leq C(\|f\|_{H^s}^{q-1}+\|g\|_{H^s}^{q-1})\|f-g\|_{H^s} \\
\label{chemsob} \|f\nabla \Delta^{-1}f - g \nabla \Delta^{-1}g \|_{H^s} \leq C(\|f\|_{H^s}+\|g\|_{H^s})\|f-g\|_{H^s} \\
\label{powmom} \|f^q-g^q\|_{M_n} \leq C(\|f\|_{H^s}^{q-1}+\|g\|_{H^s}^{q-1})\|f-g\|_{M_n} \\
\label{chemmom} \|f\nabla \Delta^{-1}f - g \nabla \Delta^{-1}g \|_{M_n} \leq C(\|f\|_{H^s}+\|g\|_{H^s}+\|g\|_{M_n})(\|f-g\|_{M_n}+\|f-g\|_{H^s}).
\end{eqnarray}
All constants in the inequalities may depend only on $q,d,s$ and $n.$
\end{lemma}
\begin{proof}
All these estimates are fairly straightforward. The estimate \eqref{powsob} follows from writing $f^q-g^q = (f-g)(f^{q-1}+\dots+g^{q-1})$
and the fact that $H^s$ is an algebra when $s>d/2$ (see, e.g. \cite{Ziem}). The estimate \eqref{chemsob} follows from a similar argument.
The third inequality \eqref{powmom} is proved  by the same expansion and use of Sobolev imbedding implying $\|f\|_{L^\infty}+\|\nabla f\|_{L^\infty} \leq C\|f\|_{H^s}$
and similar bounds for $g.$ Finally, to prove the last inequality \eqref{chemmom}, write
\[ f\nabla \Delta^{-1}f - g\nabla \Delta^{-1}g = (f-g)\nabla \Delta^{-1}f + g (\nabla \Delta^{-1} f - \nabla \Delta^{-1}g). \]
Integral of the right hand side expression against $(1+|x|^n)$ does not exceed
\[ \|f-g\|_{M_n}\|\nabla \Delta^{-1} f\|_{L^\infty} + \|g\|_{M_n} \|\nabla \Delta^{-1}(f-g)\|_{L^\infty} \leq
C(\|f\|_{H^s}+\|g\|_{M_n})(\|f-g\|_{M_n}+\|f-g\|_{H^s}). \]
For the case of the gradient, observe that
\[ \nabla \cdot(f\nabla \Delta^{-1}f - g\nabla \Delta^{-1}g) = (\nabla f\cdot\nabla \Delta^{-1}f -
\nabla g\cdot\nabla \Delta^{-1}g)+(f^2-g^2). \]
The first two terms are then controlled similarly to the previous estimate, while the last two terms are easy to handle.
\end{proof}
Now we can prove a key Lemma setting up contraction mapping.
\begin{lemma}\label{keycontr}
Suppose that $u \in C^\infty(\Rm^d \times [0,\infty))$ and $\nabla \cdot u =0.$ Let $s,q$ and $n$ be positive integers, $s>\frac{d}{2}+1.$
Let $f,g \in X^T_{s,n}.$ Then
\begin{equation}\label{contr1}
\|B_T(f)-B_T(g)\|_{X^T_{s,n}} \leq \alpha \|f-g\|_{X^T_{s,n}},
\end{equation}
where for $T \leq 1,$ we have
\begin{equation}\label{alphacon}
\alpha \leq C(d,q,n,\epsilon,\chi)\max_{0 \leq t \leq T}\left(\|u(\cdot,t)\|_{C^s}+\|f(\cdot,t)\|_{K_{s,n}}^{q-1}+\|g(\cdot,t)\|_{K_{s,n}}^{q-1}+\|f(\cdot,t)\|_{K_{s,n}}+\|g(\cdot,t)\|_{K_{s,n}}\right)T^{1/2}.
\end{equation}
\end{lemma}
\begin{proof}
Consider
\[ B_t(f)-B_t(g) = \int_0^t e^{\Delta(t-r)} \left( \nabla (u(f-g)) - \epsilon (f^q-g^q)+\chi\nabla(f \nabla \Delta^{-1}f -g\nabla \Delta^{-1}g) \right)\,dr. \]
Using Lemmas~\ref{elem1} and \ref{Blem}, we find
\begin{eqnarray}
&&\|B_t(f)-B_t(g)\|_{K_{s,n}} \leq
C \int_0^t \left[ \left((t-r)^{-1/2}+(t-r)^{(n-1)/2}\right)\left(\|u\|_{C^s}+\|f\|_{K_{s,n}}+\|g\|_{K_{s,n}}\right)
\right. \nonumber\\
&& +\left.  \left(1+(t-r)^{n/2}\right)\left(\|f\|_{K_{s,n}}^{q-1}+\|g\|_{K_{s,n}}^{q-1}\right) \right] \|f-g\|_{K_{s,n}}\,dr
\nonumber\\
&&\le C\left[ \left( t^{1/2}+t^{(n+1)/2} \right) {\rm max}_{0 \leq r \leq t}\left( \|u\|_{C^s}+\|f\|_{K_{s,n}}+\|g\|_{K_{s,n}}\right)
 \right.
 \nonumber \\  &&
 +\left. \left(t+t^{(n+2)/2}\right){\rm max}_{0 \leq r \leq t}\left(\|f\|_{K_{s,n}}^{q-1}+\|g\|_{K_{s,n}}^{q-1}\right) \right] {\rm max}_{0 \leq r \leq t}
\|f-g\|_{K_{s,n}}. \label{lastal}
\end{eqnarray}
Therefore, we obtain \eqref{contr1}. For $T \leq 1$ we can ignore the higher powers of $T$ and the estimate \eqref{alphacon} for $\alpha$ follows from
\eqref{lastal}.
\end{proof}
In a standard way, Lemma~\ref{keycontr} implies existence of local solution via the contraction mapping principle.
\begin{theorem}\label{locexthm}
Assume $q,s,n$ are positive integers and $s>\frac{d}{2}+1,$ $u \in C^\infty(\Rm^d \times [0,\infty)),$ $\nabla \cdot u=0.$ Suppose $\rho_0 \in K_{s,n}.$
Then there exists $T=T(q,d,u,s,\epsilon,\chi,\|\rho_0\|_{K_{s,n}})$ such
that there exists a unique solution $\rho(x,t) \in X^T_{s,n}$ of the equation \eqref{duhamel} satisfying $\rho(x,0)=\rho_0(x).$
\end{theorem}
\it Remark. \rm Higher regularity of the solution in space and time (in particular implying $\rho(x,t) \in C(H^m, (0,T])$ for every $m>0$) follows from
Theorem~\ref{locexthm} and standard parabolic regularity estimates applied iteratively.
\begin{corollary}\label{hscon}
If under conditions of Theorem~\ref{locexthm} we prove global a-priori estimate on $\|\rho(\cdot,t)\|_{H^s}$ and $\|\rho(\cdot,t)\|_{M_n},$ then the local solution can be
extended globally to $X^T_{n,s}$ with arbitrary $T.$
\end{corollary}
Indeed, if there is a control on growth $H^s$ and $M_n$ norms of the solution, we can just extend it by iterative application of local result as far as we want.
To prove the bounds on $M_n$ and $H^s$ norms of the solution, we first establish control of the $L^\infty$ norm.
\begin{lemma}\label{linfty}
Assume that $\rho(x,t)$ is the local solution guaranteed by Theorem~\ref{locexthm}. Then
\begin{equation}\label{linfrho}
\|\rho(\cdot,t)\|_{L^\infty} \leq N_0 \equiv {\rm max}\left( (\chi/\epsilon)^{\frac{1}{q-2}}, \|\rho_0\|_{L^\infty} \right)
\end{equation}
for all $0 \leq t \leq T.$
\end{lemma}
\begin{proof}
Assume this is false, and there exists $N_1>N_0$ and $0<t_1 \leq T$ such that we have $\|\rho(x,t_1)\|_{L^\infty} = N_1$ for the first time
(that is, for all $x$ and $0 \leq t \leq t_1,$ $|\rho(x,t_1)| \leq N_1$). We claim that in this case there exists $x_0$ such that $\rho(x_0,t_1)=N_1.$
Indeed, the only alternative is that there exists a sequence $x_k$ such that $\rho(x_k,t_1) \rightarrow N_1$ as $k \rightarrow \infty.$
If $x_k$ has finite accumulation points, set one of them as $x_0.$ By continuity $\rho(x_0,t_1)$ will be equal to $N_1.$ Thus it remains to consider
the case where $x_k \rightarrow \infty$ and passing to a subsequence if necessary we can assume that unit balls around $x_k,$ $B_1(x_k),$ are disjoint.
By a version of Poincare inequality (see e.g. \cite{Ziem}), we have $\|\rho - \overline{\rho}\|^2_{L^\infty(B_1(x_k))} \leq C\|\rho\|^2_{H^s(B_1(x_k))}.$
Since $\sum_k \|\rho\|^2_{H^s(B_1(x_k))} \leq C(t_1) < \infty,$ we get that
\[ \overline{\rho}_k \equiv \frac{1}{|B_1(x_k)|}\int_{B_1(x_k)}\rho\,dx \stackrel{k \rightarrow \infty}{\longrightarrow} N_1. \]
But this is a contradiction with $\int_{\Rm^d} |\rho(x)|(1+|x|^n)\,dx \leq C(t_1).$

Therefore, there exists $x_0$ such that $\rho(x_0,t_1) = N_1$
(we consider the case of a maximum; the case of minimum equal to $-N_1$ is considered similarly).
Then
\begin{eqnarray*}
 &&\left. \partial_t \rho(x_0,t) \right|_{t=t_1} =  (u \cdot \nabla) \rho(x_0,t_1) + \Delta \rho(x_0,t_1) + \chi \nabla \rho(x_0,t_1) \cdot
\nabla \Delta^{-1}\rho(x_0,t_1) \\
&&~~~~~~~~~~~~+ \chi \rho(x_0,t_1)^2-\epsilon \rho(x_0,t_1)^q \leq \rho(x_0,t_1)^2(\chi - \epsilon \rho(x_0,t_1)^{q-2}).
\end{eqnarray*}
By assumption on $N_1,$ we see that $\partial_t \rho(x_0,t_1) <0,$ contradiction with our choice of $t_1.$
\end{proof}

Let us first prove an upper bound on the growth of the $M_n$ norm of the solution.
\begin{lemma}\label{moments}
Assume that $\rho(x,t)$ is the local solution guaranteed by Theorem~\ref{locexthm}. Then on the interval of existence, we have the following
bound for the growth of the $M_n$ norm of the solution.
\begin{eqnarray}\label{momfineq} \|\rho(\cdot,t)\|_{M_n} \leq C(1+t^{n/2})\|\rho_0\|_{M_n} \exp\left(C\int_0^t \left[\epsilon (1+r^{n/2})\|\rho(\cdot,r)\|_{L^\infty}^{q-1}
 + \right. \right. \\ \nonumber \left. \left. (r^{-1/2}+r^{(n-1)/2})(\|u(\cdot,r)\|_{C^1}+\|\rho(\cdot, r)\|_{L^\infty}
+\|\rho(\cdot,r)\|_{L^1})\right]\,dr\right). \end{eqnarray}
\end{lemma}
\begin{proof}
Consider \eqref{duhamel}. By Lemma~\ref{elem1}, we have $\|e^{t\Delta}\rho_0\|_{M_n} \leq C(1+t^{n/2})\|\rho_0\|_{M_n}.$ Also, note that
$\|u \rho\|_{M_n} \leq \|u\|_{C^1} \|\rho\|_{M_n},$
$\|\rho^q\|_{M_n} \leq \|\rho\|_{L^\infty}^{q-1}\|\rho\|_{M_n},$ 
and \[ \|\rho \nabla \Delta^{-1} \rho \|_{M_n} \leq C\|\nabla \Delta^{-1}\rho\|_{L^\infty}\|\rho\|_{M_n} \leq C(\|\rho\|_{L^\infty}
+\|\rho\|_{L^1})\|\rho\|_{M_n}. \]
Therefore, applying these estimates and Lemma~\ref{elem1} to \eqref{duhamel}, we obtain
\begin{eqnarray}\label{mom12} \|\rho(\cdot,t)\|_{M_n} \leq C(1+t^{n/2})\|\rho_0\|_{M_n} + C\int_0^t \left( \left[\epsilon (1+r^{n/2})\|\rho(\cdot,r)\|_{L^\infty}^{q-1}
+ \right. \right. \\  \nonumber \left. \left. (r^{-1/2}+r^{(n-1)/2})(\|u(\cdot,r)\|_{C^1}+\|\rho(\cdot, r)\|_{L^\infty}
+\|\rho(\cdot,r)\|_{L^1})\right]\|\rho(\cdot,r)\|_{M_n}\right)\,dr. \end{eqnarray}
The inequality \eqref{mom12} implies the bound \eqref{momfineq}.
This upper bound is not optimal, but it is sufficient for our purpose. 
\end{proof}

Now we are ready to prove uniform in time bounds on the $H^s$ norm of the solution.
\begin{lemma}\label{unsoblem}
Let $\rho(x,t)$ be the local solution whose existence is guaranteed by Theorem~\ref{locexthm}.
Suppose that $\|\rho(\cdot,t)\|_{L^\infty}$ does not exceed $N_0$ for all $0 \leq t \leq T.$
Then
\begin{equation}\label{unsob}
\|\rho(\cdot,t)\|_{H^s} \leq {\rm max}\left(\|\rho_0\|_{H^s}, C(u,d,q,s,\chi,\epsilon,N_0)\right).
\end{equation}
\end{lemma}
\begin{proof}
Consider for simplicity the case where $s$ is even (the odd case is very similar).
Apply $\Delta^{s/2}$ to \eqref{chemnew}, multiply by $\Delta^{s/2}\rho(x,t)$ and integrate.
We obtain
\begin{eqnarray}\label{sobdyn}
&&\!\!\!\!\!\!\!\!\!\!\!\!\!\!\!\!\!\!\!\!\!\!\!\!
\frac12 \partial_t \|\rho\|_{H^s}^2 = \int\limits_{\Rm^d}[\Delta^{s/2} (u\cdot \nabla)\rho]( \Delta^{s/2}\rho)\,dx
- \epsilon \int\limits_{\Rm^d} (\Delta^{s/2}\rho^q)( \Delta^{s/2} \rho)\,dx\nonumber
-\|\rho\|^2_{H^{s+1}} \\
 && + \chi\int_{\Rm^d}[\nabla \cdot\Delta^{s/2} (\rho \nabla \Delta^{-1}\rho) ](\Delta^{s/2}\rho)\,dx.
\end{eqnarray}
Using $\nabla \cdot u =0,$ we obtain
\[
\left| \int_{\Rm^d}[\Delta^{s/2} ((u\cdot \nabla)\rho)]( \Delta^{s/2}\rho)\,dx \right| \leq C\|u\|_{C^s}\|\rho\|^2_{H^s}.
\]
Next, the second integral on the right hand side of \eqref{sobdyn} can be written as a sum of a finite number of terms of the form
$\int_{\Rm^d} D^s \rho \prod_{i=1}^q D^{s_i}\rho \,dx,$ $s_1+\dots+s_q = s,$ $s_i \geq 0.$ Here $D^l$ denotes any partial derivative
operator of the $l$th order. By H\"older's inequality, we have
\[ \left| \int_{\Rm^d} D^s \rho \prod_{i=1}^q D^{s_i}\rho \,dx \right| \leq \|D^s \rho\|_{L^2} \prod_{i=1}^q \|D^{s_i}\rho\|_{p_i}, \]
$\sum_{i=1}^q p_i^{-1} =1/2$. Take $p_i = 2s/s_i,$ and recall the Gagliardo-Nirenberg inequality (\cite{Gag,Nir,Maz})
\begin{equation}\label{gn1} \|D^{s_i}\rho\|_{L^{2s/s_i}} \leq C \|\rho\|_{L^\infty}^{1-\frac{s_i}{s}}\|D^s\rho\|_{L^2}^{\frac{s_i}{s}}. \end{equation}
Then we get
\[ \left| \int_{\Rm^d} \Delta^{s/2}\rho^q \Delta^{s/2} \rho\,dx \right| \leq C \|\rho\|_{L^\infty}^{q-1}\|\rho\|_{H^s}^2. \]
Finally, we claim that the third integral on the right hand side of \eqref{sobdyn} can be written as a sum of a finite number of terms
of the form $\int_{\Rm^d} D^s \rho D^k \rho D^{s+2-k} \Delta^{-1}\rho\,dx,$ where $k = 0,\dots,s.$
The only term one gets from the direct
differentiation that does not appear to be of this form is $\int_{\Rm^d} \Delta^{s/2}\rho \nabla \Delta^{s/2} \rho \nabla \Delta^{-1}\rho\,dx.$
However, integrating by parts, we find that this term is equal to $-\frac12
\int_{\Rm^d} |\Delta^{s/2}\rho|^2 \rho\,dx.$ Now
\[ \left| \int_{\Rm^d} D^s \rho D^k \rho D^{s+2-k} \Delta^{-1}\rho\,dx \right| \leq C \|D^s \rho\|_{L^2} \|D^k \rho\|_{L^{p_1}} \|D^{s-k}\rho\|_{L^{p_2}}, \]
$p_1^{-1}+p_2^{-1}=1/2,$ $p_2 < \infty.$ Here we used boundedness of Riesz transforms on $L^{p_2},$ $p_2 <\infty.$ Set $p_1 = \frac{2s}{k},$
$p_2= \frac{2s}{s-k}.$ By Gagliardo-Nirenberg inequality \eqref{gn1} with $s_i=k,s-k,$ we get
\[ \left| \int_{\Rm^d} D^s \rho D^k \rho D^{s+2-k} \Delta^{-1}\rho\,dx \right| \leq C \|\rho\|_{L^\infty} \|\rho\|_{H^s}^2. \]
Putting all the estimates into \eqref{sobdyn}, we find that
\begin{equation}\label{disn} \frac12 \partial_t \|\rho\|^2_{H^s} \leq C \|\rho\|_{L^\infty} \|\rho\|^2_{H^s} - \|\rho\|^2_{H^{s+1}} \leq C \|\rho\|_{L^\infty} \|\rho\|^2_{H^s} -
\|\rho\|_{H^s}^{2+\frac{2}{s-d/2}}\|\rho\|_{L^\infty}^{-\frac{2}{s-d/2}}. \end{equation}
 We used another Gagliardo-Nirenberg inequality in the last step.
The differential inequality \eqref{disn} implies the result of the lemma.
\end{proof}
Given the $M_n$ and $H^s$ norm bounds we proved, the solution can now be continued globally, completing the proof of Theorem~\ref{globex}.

\section{Appendix II: The Gagliardo-Nirenberg inequality with $p<1$}\label{gnp}

Twice in the paper, we needed to apply Gagliardo-Nirenberg inequalities with one of the summation exponents less than one
(see \eqref{nonstgn}, \eqref{nonstgn1}). Such inequalities are certainly known and can be found in mathematical literature (see e.g.
encyclopedic \cite{Maz}). However, it was not easy for us to find a reference with a transparent self-contained proof, and for the sake
of completeness we provide a sketch of an elegant and simple proof here.
The idea of this argument has been communicated to us by Fedor Nazarov.
We will prove a slightly more general inequality containing both \eqref{nonstgn} and \eqref{nonstgn1}.

\begin{theorem}\label{nonstgn0}
Let $v \in C_0^\infty(\Rm^d),$ $d \geq 2.$ Then
\begin{equation}\label{gengn}
\|v\|_{L^q} \leq C(q,d)\|\nabla v\|_{L^2}^a \|v\|_{L^r}^{1-a}, \,\,\, a=\frac{\frac{1}{r}-\frac{1}{q}}{\frac{1}{d}-\frac12+\frac{1}{r}}.
\end{equation}
The inequality holds for all $q,r>0$ such that $q>r$ and $\frac{1}{d}-\frac12 + \frac{1}{r}>0.$
\end{theorem}
\begin{proof}
Let $A_k$ denote regions in $\Rm^d$ such that $|A_k|=2^{kd},$ $k \in \Zm,$ the boundary of $A_k$ coincides with a level set of $|v(x)| \equiv v_{k+1},$ and
$|v(x)| \geq v_{k+1}$ inside $A_k.$ 
Then
\[ \|v\|_{L^q}^q \leq \sum_{k \in \Zm} |A_k| v_k^q. \]
Fix a small $\delta>0.$ Let us call $k$ "important" if $v_{k+1} < (1-\delta)v_k.$
Denote the set of all important $k$ by $I.$ Observe that
\[ \sum_{k \in \Zm} |A_k| v_k^q \leq C(\delta) \sum_{k \in I} |A_k| v_k^q. \]
Indeed, a sequence of not important consequent $k$ contributes at most $\sum_{l >0} 2^{-dl}(1-\delta)^{-ql}|A_{k+1}|v_{k+1}^q$
compared to the contribution
$|A_{k+1}|v_{k+1}^q$ of the single next term.

For the $L^r$ norm, we have the estimate
\[ \|v\|_{L^r}^{(1-a)q} \geq C \left( \sum_{k \in \Zm} |A_k| v_k^r \right)^{(1-a)q/r}. \]
For the gradient term, by the co-area formula (see e.g. \cite{EG}) we have
\[ \int_{v_{k+1} \leq v(x) \leq v_k} |\nabla v| \,dx = \int_{v_k}^{v_{k+1}} \cH^{d-1}(x:\,|v(x)|=s)\,ds, \]
where $\cH^{d-1}$ is the $d-1$-dimensional Hausdorff measure. By the isoperimetric inequality,
\[ \cH^{d-1}(x:\,|v(x)|=s) \geq C |A_k|^{1-\frac{1}{d}} \]
if $s \geq v_{k+1}$ (see e.g. \cite{EG}). Therefore,
\[  \int_{v_{k+1} \leq v(x) \leq v_k} |\nabla v| \,dx \geq C |A_k|^{1-\frac{1}{d}}(v_k - v_{k+1}). \]
By Cauchy-Schwartz,
\[  \int_{v_{k+1} \leq v(x) \leq v_k} |\nabla v|^2 \,dx \geq \frac{1}{|A_k|}\left(\int_{v_{k+1} \leq v(x) \leq v_k} |\nabla v| \,dx\right)^2 \geq
C|A_k|^{1-\frac{2}{d}}(v_k-v_{k+1})^2. \]
Therefore,
\[ \int_{\Rm^d} |\nabla v|^2\,dx \geq C\sum_{k \in \Zm} (v_k -v_{k+1})^2 |A_k|^{1-\frac{2}{d}} \geq C\delta^2 \sum_{k \in I} v_k^2 |A_k|^{1-\frac{2}{d}}. \]

Thus, it remains to prove that
\begin{equation}\label{gnfinhold}
\sum_{k \in I} |A_k| v_k^q \leq C\left( \sum_{k \in \Zm} |A_k| v_k^r \right)^{(1-a)q/r} \left(\sum_{k \in I} v_k^2 |A_k|^{1-\frac{2}{d}}\right)^{aq/2}.
\end{equation}
Observe that, if $d\ge 3$, then we have
\[ \left(\sum_{k \in I} v_k^2 |A_k|^{1-\frac{2}{d}}\right)^{aq/2} \geq \left( \sum_{k \in I} v_k^{\frac{2d}{d-2}}|A_k| \right)^\frac{aq(d-2)}{2d} \]
(since $\sum_k b_k^s \geq (\sum_k b_k)^s$ for $b_k \geq 0,$ $0 < s \leq 1$).
Write
\begin{equation}\label{brack} |A_k| v_k^q = \left[ |A_k|^{(1-a)q/r} v_k^{(1-a)q} \right] \left[ v_k^{aq} |A_k|^{\frac{aq(d-2)}{2d}} \right]. \end{equation}
Apply H\"older inequality on the left hand side of \eqref{gnfinhold}, rasing the first term in \eqref{brack} to the power $\frac{r}{q(1-a)},$ and the
second term to the power $\frac{2d}{aq(d-2)}.$ Notice that the inverses of these powers sum to one due to the definition of $a$ in \eqref{gengn}.
The resulting inequality coincides with \eqref{gnfinhold}. Finally, when $d=2$, we have $a=1-q/r$, and (\ref{gnfinhold}) follows
from a more elementary consideration.
\end{proof}

\noindent {\bf Acknowledgement.} \rm AK is grateful to Stanford University for its hospitality
in May 2010, when this work was initiated. AK acknowledges support of the NSF grant DMS-0653813, and thanks
the Institute for Mathematics and Its Applications for stimulating atmosphere at the Workshop on Transport and Mixing in Complex and
Turbulent Flows in March 2010, where he learned about coral spawning from the talk of Prof. Jeffrey Weiss.
LR is supported in part by NSF grant DMS-0908507. We are grateful to the referees for useful comments and corrections.  

\end{document}